\documentclass[12pt, reqno]{amsart}

\usepackage[round]{natbib}
\usepackage[T1]{fontenc}
\usepackage{amsmath}
\usepackage{amssymb}
\usepackage{amsthm}
\usepackage{fancyhdr}
\usepackage{enumitem}
\usepackage{comment}
 \usepackage{relsize}
\usepackage{nicefrac}
\usepackage{bm}
\usepackage{mathrsfs}
\usepackage{graphicx}
\usepackage[utf8]{inputenc}
\usepackage{cancel}
\usepackage{mathtools}
\usepackage{tikz}
\usepackage{dirtytalk}
\usetikzlibrary{cd,decorations.pathreplacing,matrix,arrows,positioning,automata,shapes,shadows,calc,fadings,decorations,through,intersections}
\usepackage{float}

\AtBeginDocument{%
   \def\MR#1{}
}

\newtheorem{theorem}{Theorem}
\newtheorem{corollary}{Corollary}
\newtheorem{lemma}{Lemma}
\newtheorem{proposition}{Proposition}
\newtheorem{conj}{Conjecture}

\theoremstyle{definition} 
\newtheorem{defi}{Definition}

\newtheorem{question}{Question}
\let\oldquestion\question
\renewcommand{\question}{\oldquestion\normalfont}
\newtheorem{example}{Example}
\let\oldexample\example
\renewcommand{\example}{\oldexample\normalfont}
\newtheorem{remark}{Remark}
\let\oldrmk\remark
\renewcommand{\remark}{\oldrmk\normalfont}
\newtheorem*{claim}{\textsc{Claim}}

\usepackage{todonotes}

\usepackage[left=2.4cm,top=2.8cm,right=2.4cm,bottom=2.5cm,head=3cm,headsep=1cm, foot=3cm]{geometry}

\usepackage[bottom,flushmargin,hang,multiple]{footmisc}
\usepackage{setspace}
\usepackage{xr-hyper}
\usepackage{hyperref}
\usepackage{commath}
\definecolor{darkblue}{rgb}{0.1,0.1,0.9}
\definecolor{darkred}{rgb}{0.9,0.1,0.1}

\hypersetup{colorlinks, allcolors= darkblue}

\allowdisplaybreaks
\makeatletter

\newcommand{\Rmnum}[1]{\expandafter\@slowromancap\romannumeral #1@}
\makeatother

\providecommand{\MR}[1]{}

\providecommand{\MR}{\relax\ifhmode\unskip\space\fi MR }

\providecommand{\href}[2]{#2}

\begin{document}

\title[A Nonlinear Sandwich Theorem]{A Nonlinear Sandwich Theorem}

\author[Mario Ghossoub, Giulio Principi, and Lorenzo Stanca]{Mario Ghossoub\vspace{0.1cm}\\University of Waterloo\vspace{0.6cm}\\Giulio Principi\vspace{0.1cm}\\New York University\vspace{0.6cm}\\Lorenzo Stanca\vspace{0.1cm}\\Collegio Carlo Alberto and University of Turin\vspace{1cm}\\ \today\vspace{-0.2cm}}

\address{{\bf Mario Ghossoub}: University of Waterloo -- Department of Statistics and Actuarial Science -- 200 University Ave.\ W.\ -- Waterloo, ON, N2L 3G1 -- Canada}
\email{\href{mailto:mario.ghossoub@uwaterloo.ca}{mario.ghossoub@uwaterloo.ca}\vspace{0.2cm}}

\address{{\bf Giulio Principi}: New York University -- Department of Economics -- 19 West 4th Street -- New York, NY 10012 -- USA}
\email{\href{mailto:gp2187@nyu.edu}{gp2187@nyu.edu}\vspace{0.2cm}}

\address{{\bf Lorenzo Stanca}: Collegio Carlo Alberto and University of Turin -- Department of ESOMAS -- Corso Unione Sovietica, 218 Bis -- Turin 10134 -- Italy}
\email{\href{mailto:lorenzo.stanca@carloalberto.org}{lorenzo.stanca@carloalberto.org}}

\thanks{We thank Andrea Aveni, Fabio Maccheroni, and Ren\'e Pfitscher for useful discussions, as well as Stephen Simons for providing some relevant references. Mario Ghossoub acknowledges financial support from the Natural Sciences and Engineering Research Council of Canada (NSERC) (Grant No.\ 2018-03961). Giulio Principi is grateful for the financial support provided by the Henry M.\ MacCracken Fellowship at New York University.\vspace{0.2cm}}

\keywords{Sandwich Theorem; Pataraia's Theorem; $\textbf{C}$-(super/sub)linearity; convex-conic symmetric preorder\vspace{0.2cm}}
\subjclass[2020]{Primary: 46A22; Secondary: 06A06, 06F20.\vspace{0.2cm}}


\maketitle


\begin{abstract}
We provide a Sandwich Theorem (\cite{Konig72}) for positively homogeneous functionals that satisfy additivity only on a restricted domain. Our relaxation of additivity is based on a binary relation called \textit{convex-conic symmetric preorder}, whereby additivity is restricted to all couples of elements that belong to such relation. We then study applications of our nonlinear Sandwich Theorem, proving extension and envelope representation results. Finally, we consider some applications to comonotonicity, a key property in decision theory, risk measurement, and the theory of risk sharing.
\end{abstract}

\vspace{0.4cm}

\section{Introduction}
The classic version of the Sandwich Theorem yields the existence of a linear functional in between a given sublinear functional and a given superlinear functional. \cite{Konig72} proved this result as a corollary to the Hahn-Banach extension theorem. In a recent paper, \cite{Amarante_positivity} provided an alternative proof, 
by combining an argument in \cite{Isotone_Wright} with Pataraia's Fixed Point Theorem (\cite{pataraia1997constructive}). In this paper, we show how the elegant approach developed by \cite{Amarante_positivity} can be directly implemented to retrieve nonlinear versions of the Sandwich Theorem, restricting (super/sub)linearity only to some portions of the domain of the functionals involved. Specifically, we define a binary relation (denoted by $\mathbf{C}$) that we call the \textit{convex-conic symmetric preorder}, which is a symmetric preorder closed with respect to positive scalar multiplication and addition. Using this preorder, we provide a Sandwich Theorem (see Theorem \ref{prop_Pataraia}) for $\mathbf{C}$-(super/sub)linear functionals, i.e., functionals that satisfy positive homogeneity and (super/sub)additivity only with respect to $(x,y)\in \mathbf{C}$. We then use this result to prove a Hahn-Banach type extension result (see Corollary \ref{Ext_Theorem}) and an envelope representation result (see Corollary \ref{envelope_Theorem}). We conclude by providing an illustration of the applicability of our results, with a focus on comonotonic subadditive functionals.

\vspace{0.4cm}
\section{Pataraia's Theorem and Order-Theoretic Terminology}

\subsection{Order-Theoretic Terminology}\footnote{We summarize in this subsection all the order-theoretic concepts that we use in this paper. 
For a comprehensive treatment of these notions, we refer the reader to \cite{Finite_ordered_set} and \cite{Schroder_ordered_sets}.} Fix a nonempty set $S$. By a \textit{(binary) relation} over $S$ we mean a set $\mathbf{R}\subseteq S\times S$. For all $x,y,z\in S$, we will often write $x\mathbf{R} y$ in place of $(x,y)\in \mathbf{R}$, and $x\mathbf{R} y \mathbf{R} z$ instead of $x\mathbf{R}y$ and $y\mathbf{R}z$. 

\vspace{0.2cm}

We say that a binary relation $\mathbf{R}$ is \textit{reflexive} if $x\mathbf{R}x$ for all $x\in S$, while it is \textit{symmetric} if $x\mathbf{R}y$ implies $y\mathbf{R}x$, for all $x,y\in S$. If $x\mathbf{R}y$ and $y\mathbf{R}z$ implies $x\mathbf{R}z$, for all $x,y,z\in S$, then $\mathbf{R}$ is said to be \textit{transitive}. A reflexive and transitive relation is called a \textit{preorder}. 
A preorder $\mathbf{R}$ is a \textit{partial order} if it is also \textit{antisymmetric}, that is, if $x\mathbf{R}y$ and $y\mathbf{R}x$, then $x=y$, for all $x,y\in S$, in which case we say that $(S,\mathbf{R})$ is a \textit{partially ordered set} (henceforth, \textit{poset}). A partial order $\mathbf{R}$ is said to be a \textit{total order} if for all $x,y\in S$, either $x\mathbf{R}y$ or $y\mathbf{R}x$, in which case we say that $(S,\mathbf{R})$ is a \textit{totally ordered set}. A totally ordered subset of a poset will be referred to as a \textit{chain}. 

\vspace{0.2cm}

Given a poset $(S,\mathbf{R})$ we say that an element $x\in S$ is an $\mathbf{R}$-\textit{upper bound} of $B\subseteq S$ if $x\mathbf{R}B$, that is, $x\mathbf{R}y$ for all $y\in B$. We say that $x$ is an $\mathbf{R}$-\textit{supremum} for $B\subseteq S$ if it is an $\mathbf{R}$-upper bound of $B$ and $z\mathbf{R}x$, for all $z\in S$ with $z\mathbf{R}B$. Infima are defined analogously. Given a poset $(S,\mathbf{R})$, we denote $\mathbf{R}$-suprema and $\mathbf{R}$-infima of $B\subseteq S$ by $\mathbf{R}$-$\sup B$ and $\mathbf{R}$-$\inf B$, respectively. When the binary relation is well-understood within the context we will simply write $\sup$ and $\inf$ to ease the notation. We say that a poset $(S,\mathbf{R})$ is \textit{strictly inductively ordered} if every chain $C\subseteq S$ admits a $\mathbf{R}$-supremum. Given a poset $(S,\mathbf{R})$, we say that a mapping $F:S\to S$ is $\mathbf{R}$-\textit{inflationary} if $F(s)\mathbf{R}s$, for all $s\in S$. Given a nonempty set $A$, a poset $(S,\mathbf{R})$, and two maps $f,g:A\to S$, we say that $f\mathbf{R}g$ if and only if $f(a)\mathbf{R}g(a)$ for all $a\in A$. We will refer to this order as the $\mathbf{R}$-\textit{pointwise order}, and, often simply as the \textit{pointwise order}, when the underlying relation is clear from the context. It is immediately noticeable that given a poset $(S,\mathbf{R})$ and a set $A$, the induced $\mathbf{R}$-pointwise order on $S^A$ is also a partial order.

\vspace{0.2cm}

We say that a (real) vector space $V$ is an \textit{ordered vector space} if it is endowed with a partial order $\leq$, such that $x\leq y$ if and only if $x+z\leq y+z$ and $\alpha x\leq \alpha y$, for all $x,y,z\in V$ and all $\alpha\geq 0$.\footnote{We will exclusively focus on real vector spaces, which we will henceforth refer to simply as vector spaces.} In addition, we say that an ordered vector space $(V,\leq)$ is a \textit{Riesz space} if it is also a \textit{lattice}, that is, $\sup\left\lbrace x,y\right\rbrace,\inf\left\lbrace x,y \right\rbrace\in V$, for all $x,y\in V$. A subset $B$ of a Riesz space $(V,\leq)$ is said to be \textit{bounded from above} if it admits an upper bound in $V$. Finally, we say that a Riesz space $(V,\geq)$ is \textit{Dedekind-complete} if each subset of $V$ which is bounded from above admits a supremum.\footnote{We refer to \cite{Positive_operators} for a detailed treatment of Riesz spaces.} For any vector space $V$ we will denote by $\mathbf{0}_{V}$ its null element.

\vspace{0.2cm}
\subsection{Pataraia's Fixed Point Theorem}
Similarly to \cite{Amarante_positivity} we will apply the following version of Pataraia's Fixed Point Theorem (see also \cite{escardo}).

\vspace{0.1cm}

\begin{theorem}[Pataraia]\label{pataraiaFPT}
 Let $(S, \mathbf{R})$ be a poset and $\mathcal{I}$ be the set of all inflationary mappings on $S$. If $(S,\mathbf{R})$ is strictly inductively ordered, then $\mathcal{I}$ has a common fixed point, that is, there exists $x\in S$ such that $f(x)=x$, for every $f\in\mathcal{I}$.
\end{theorem}

\vspace{0.1cm}

\begin{proof}
Since $\mathbf{R}$ is reflexive, $\mathrm{Id}\in \mathcal{I}$. Moreover, given that all $f\in \mathcal{I}$ are inflationary, it follows that $f\mathbf{R} \mathrm{Id}$. Suppose that $\mathcal{C}$ is a chain in $\mathcal{I}$ and define $\overline{f}:s\mapsto \sup\left\lbrace f(s):f\in \mathcal{C} \right\rbrace$. Clearly $\overline{f}\in \mathcal{I}$ and hence it is a $\mathbf{R}$-supremum for $\mathcal{C}$ with respect to the pointwise order. This yields that $\mathcal{I}$ with the pointwise order is strictly inductively ordered. Therefore, by Zorn's Lemma, $\mathcal{I}$ has a maximal element, say $M\in \mathcal{I}$. Note that for all $s\in S$ and all $f\in \mathcal{I}$, $f(M(s))\mathbf{R} M(s)$. Thus, since $M$ is maximal and $\mathbf{R}$ is antisymmetric, we must have $f\circ M=M$. Therefore, for all $s\in S$, $M(s)$ is a common fixed point of $\mathcal{I}$.
\end{proof}

\vspace{0.4cm}
\section{Main Results}

\subsection{Convex-Conic Symmetric Preorders}
Given a vector space $V$, we say that $X\subseteq V$ is a \textit{convex cone} if $\lambda X\subseteq X$ for all $\lambda>0$, and $X+X\subseteq X$. Now let $X$ be a given convex cone. 

\vspace{0.1cm}

\begin{defi}\label{def:ccsp}
A binary relation $\mathbf{C}\subseteq X\times X$ is a \textit{convex-conic symmetric preorder} if it is reflexive, transitive, symmetric, and it satisfies the following properties:
\vspace{0.2cm}
\begin{enumerate}[label=(\roman*)]
    \item \label{item1:poshom_ccsp} $(\lambda x) \mathbf{C} x$, for all $x\in X$ and $\lambda>0$.
    \vspace{0.3cm}
    \item \label{item2:add_ccsp} $x\mathbf{C}y\mathbf{C}z\ \textnormal{implies}\ (x+y)\mathbf{C}z$, for all $x,y,z\in X$.
\end{enumerate}
\end{defi}

\vspace{0.2cm}

The \say{convex-conic} adjective in the definition of a convex-conic symmetric preorder is due to the fact that $\mathbf{C}(x)=\left\lbrace y\in X:y\mathbf{C}x \right\rbrace$ is a convex cone. Note that the symmetry property \ref{item2:add_ccsp} in Definition \ref{def:ccsp} implies that whenever $x\mathbf{C}y\mathbf{C}z$, we have $(g+u)\mathbf{C}h$, for all $g,u,h\in \left\lbrace x,y,z\right\rbrace$. It is important to observe that whenever $\left\lbrace \mathbf{0}_V\right\rbrace\times X\subseteq \mathbf{C}$, for some convex-conic symmetric preorder $\mathbf{C}$ on $X$, it follows that $X\times X=\mathbf{C}$. This is a straightforward consequence of symmetry and transitivity. Because of property \ref{item2:add_ccsp}, the same would hold if $(-x)\mathbf{C}x$ for all $x\in X$ with $-x\in X$. For future reference we recollect these simple observations on the following lemma.

\vspace{0.2cm}

\begin{lemma}\label{worrisome_lemma}
Let $X$ be a convex cone and $\mathbf{C}\subseteq X\times X$ a convex-conic symmetric preorder. If either $\left\lbrace \mathbf{0}_V\right\rbrace\times X\subseteq \mathbf{C}$ or $(-x)\mathbf{C}x$ for all $x\in X$ with $-x\in X$, then $X\times X=\mathbf{C}$.
\end{lemma}

\vspace{0.2cm}

Therefore, at first sight one may worry about how permissive a convex-conic symmetric preorder is. We show by means of examples (see Section \ref{examples}) that it is neither always trivial nor too restrictive.


\vspace{0.2cm}
\subsection{A Sandwich Theorem}\label{sandwich_ccsp}

Let $\left(\mathbb{V},\leq \right)$ be a Dedekind-complete Riesz space. Following \cite{Isotone_Wright}, we adjoin to $\mathbb{V}$ an element denoted by $-\infty$, and we extend $\leq$ to $\mathbb{V}\cup\left\lbrace -\infty\right\rbrace$ assuming $-\infty\leq \mathbb{V}$. We say that a map $F:X\to\mathbb{V}\cup\left\lbrace -\infty \right\rbrace$ is $\mathbf{C}$-\textit{sublinear} if
\vspace{0cm}
\begin{enumerate}[label=(\roman*)]
    \item  $F$ is \textit{positively homogeneous}, i.e., $F(\lambda x)=\lambda F(x)$, for all $\lambda > 0$ and all $x\in X$.
    \vspace{0.3cm}
    \item $F$ is $\mathbf{C}$-\textit{subadditive}, i.e., $F(x+y)\leq F(x)+F(y)$, for all $x,y\in X$ with $x\mathbf{C}y$.
\end{enumerate}

\vspace{0.3cm}

The definitions of $\mathbf{C}$-superlinearity and $\mathbf{C}$-linearity are analogous. We endow $V$ with a partial order $\preceq$ such that $(V,\preceq)$ is a partially ordered vector space. We say that a function $F:X\to \mathbb{V}\cup\left\lbrace -\infty \right\rbrace$ is \textit{monotone} if $x\preceq y$ implies $F(x)\leq F(y)$, for all $x,y\in X$. We adopt the convention that $0\cdot (-\infty)=\mathbf{0}_{\mathbb{V}}$. We are now ready to state and prove our main result.

\vspace{0.2cm}

\begin{theorem}\label{prop_Pataraia}
Suppose that $P:X\to \mathbb{V}\cup\left\lbrace -\infty \right\rbrace$ is $\mathbf{C}$-superlinear and $H:X\to \mathbb{V}\cup\left\lbrace -\infty \right\rbrace$ is $\mathbf{C}$-sublinear and monotone. If $P\leq H$, then there exists a $\mathbf{C}$-linear map $Q:X\to \mathbb{V}\cup\left\lbrace -\infty \right\rbrace$ such that $P\leq Q\leq H$. 
\end{theorem}

\vspace{0.2cm}

The proof will be provided in several steps. Before going into its details, we provide some definitions and simple remarks. Let 
\[\mathcal{D}_{\mathrm{PH}}=\left\lbrace Q:X\to \mathbb{V}\cup\left\lbrace -\infty \right\rbrace:Q\ \textnormal{is}\ \mathbf{C}\textnormal{-superlinear}\ \hbox{and} \  P\leq Q\leq H \right\rbrace.\]

\vspace{0.2cm}

\noindent Clearly, $P\in \mathcal{D}_{\mathrm{PH}}\neq \emptyset$. Moreover, $\mathcal{D}_{\mathrm{PH}}$ is a poset with respect to the pointwise order. If $\mathcal{C}$ is a chain in $\mathcal{D}_{\mathrm{PH}}$, then $\overline{Q}:x\mapsto \sup\left\lbrace Q(x):Q\in \mathcal{C} \right\rbrace$ is $\textbf{C}$-superlinear and a supremum of $\mathcal{C}$. Indeed, for all $(x,y)\in \mathbf{C}$ and all $Q\in \mathcal{C}$,
\[
\overline{Q}(x+y)\geq Q(x+y)\geq Q(x)+Q(y).
\]

\vspace{0.2cm}

\noindent This implies that $\overline{Q}$ is $\textbf{C}$-superlinear. Thus $\mathcal{D}_{\mathrm{PH}}$ is strictly inductively ordered. Let $A:Q\mapsto A_Q$ be defined as
\[
A_Q(x)=\inf\limits_{y}\left\lbrace H(x+y)-Q(y):y\in \mathbf{C}(x)\ \textnormal{and}\ Q(y)>-\infty \right\rbrace,
\]

\vspace{0.2cm}

\noindent for all $x\in X$ and all $Q\in \mathcal{D}_{\mathrm{PH}}$. Now fix $g,x\in X$ and $Q\in \mathcal{D}_{\mathrm{PH}}$. We also define,
\[
T_g(Q)(x)=\sup\limits_{h,\lambda}\left\lbrace Q(h)+\lambda A_Q(g):h+\lambda g\preceq x,\ h\in \mathbf{C}(x),\ \lambda\geq 0 \right\rbrace.
\]

\vspace{0.2cm}

\noindent We then obtain the following result.

\vspace{0.1cm}

\begin{lemma}\label{toolkit_lemma}
The following claims hold:
\vspace{0.2cm}
\begin{enumerate}
    \item For all $Q\in \mathcal{D}_{\mathrm{PH}}$, we have $A_Q\geq Q$ and $A_Q$ is $\mathbf{C}$-sublinear.
    \vspace{0.3cm}
    \item For all $Q\in \mathcal{D}_{\mathrm{PH}}$ and $g\in X$, we have $T_g(Q)$ is $\mathbf{C}$-superlinear.
    \vspace{0.3cm}
    \item For all $Q\in \mathcal{D}_{\mathrm{PH}}$, we have $Q\leq T_g(Q)\leq H$. 
      \vspace{0.3cm}
    \item For all $Q\in \mathcal{D}_{\mathrm{PH}}$ and $x\in X$, we have $T_x(Q)(x)\geq A_Q(x)$.
\end{enumerate}
\end{lemma}

\vspace{0.2cm}

\begin{proof}
\textbf{(1)}. Fix $Q\in \mathcal{D}_{\mathrm{PH}}$ and $x\in X$, then
\[
H(x+y)-Q(y)\geq Q(x+y)-Q(y)\geq Q(x)+Q(y)-Q(y)=Q(x),
\]

\vspace{0.2cm}

\noindent for all $y\in \mathbf{C}(x)$ with $Q(y)>-\infty$. Therefore, $A_Q\geq Q$. Now suppose that $(x_1,x_2)\in \mathbf{C}$ and $y_1,y_2\in \mathbf{C}(x_1)$ with $Q(y_1)>-\infty$, $Q(y_2)>-\infty$. Notice that since $y_1,y_2\in \mathbf{C}(x_1)$ and $\mathbf{C}$ is symmetric and transitive, we have $x_1\mathbf{C}y_1 \mathbf{C} y_2\mathbf{C}x_2$. By symmetry and property \ref{item2:add_ccsp}, it follows that
\[
\left((x_1+y_1),(x_2+y_2)\right)\in \mathbf{C}.
\]

\vspace{0.2cm}

\noindent These observations imply that $Q(y_1+y_2)\geq Q(y_1)+Q(y_2)$, and hence 
\begin{align*}
H(x_1+x_2+y_1+y_2)-Q(y_1+y_2)&\leq H(x_1+x_2+y_1+y_2)-Q(y_1)-Q(y_2)\\
&\leq H(x_1+y_1)-Q(y_1)+H(x_2+y_2)-Q(y_2).
\end{align*}

\vspace{0.2cm}

\noindent Therefore, $A_Q(x_1+x_2)\leq A_Q(x_1)+A_Q(x_2)$. Consequently, for all $Q\in \mathcal{D}_{\mathrm{PH}}$, $A_Q$ is $\mathbf{C}$-subadditive. Positive homogeneity follows from the positive homogeneity of $Q$ and $H$, as well as from the fact that $\mathbf{C}(x)$ is a cone.

\vspace{0.2cm}

\textbf{(2)}. Fix $Q\in \mathcal{D}_{\mathrm{PH}}$, $g\in X$, and $(x_1,x_2)\in \mathbf{C}$. Then $T_g(Q)(x_1+x_2)\geq T_g(Q)(x_1)+T_g(Q)(x_2)$. Indeed, notice that if $h_1\in \mathbf{C}(x_1)$, $h_2\in \mathbf{C}(x_2)$, $\lambda_1,\lambda_2\geq 0$ satisfy
\[
h_1+\lambda_1 g\preceq x_1\ \mathrm{and}\ h_2+\lambda_2g\preceq x_2,
\]
then $h_1+h_2+(\lambda_1+\lambda_2) g\preceq x_1+x_2$. Since $x_1\mathbf{C}x_2$, by symmetry, transitivity, and property \ref{item2:add_ccsp} we have that $x_1 \mathbf{C} x_2 \mathbf{C} h_1 \mathbf{C}h_2$ and
\[
(h_1+h_2,\lambda_1+\lambda_2)\in \left\lbrace (h,\lambda):h+\lambda g\preceq x_1+x_2,\ h\in \mathbf{C}(x_1+x_2),\ \lambda\geq 0 \right\rbrace. 
\]

\vspace{0.2cm}

\noindent This, together with the $\mathbf{C}$-superadditivity of $Q$, yields
\begin{align*}
    T_g(Q)(x_1+x_2)&=\sup\limits_{h,\lambda}\left\lbrace Q(h)+\lambda A_Q(g):h+\lambda g\preceq x_1+x_2,\ h\in \mathbf{C}(x_1+x_2),\ \lambda\geq 0 \right\rbrace\\
    &\geq \sup\limits_{h_1,h_2,\lambda_1,\lambda_2}\left\lbrace Q(h_1+h_2)+(\lambda_1+\lambda_2) A_Q(g):\begin{array}{l} h_1+h_2+(\lambda_1+\lambda_2) g\preceq x_1+x_2,\\ h_1\in \mathbf{C}(x_1),h_2\in \mathbf{C}(x_2),\lambda_1,\lambda_2\geq 0 \end{array} \right\rbrace\\
    &\geq Q(h_1)+Q(h_2)+\lambda_1  A_Q(g)+\lambda_2  A_Q(g),
\end{align*}

\vspace{0.2cm}

\noindent for all $h_1\in \mathbf{C}(x_1)$, $h_2\in \mathbf{C}(x_2)$, $\lambda_1,\lambda_2\geq 0$ with $h_1+\lambda_1 g\preceq x_1\ \mathrm{and}\ h_2+\lambda_2g\preceq x_2.$ Thus, 
\[
T_g(Q)(x_1+x_2)\geq T_g(Q)(x_1)+T_g(Q)(x_2).
\]

\vspace{0.2cm}

\noindent This proves that for all $Q\in \mathcal{D}_{\mathrm{PH}}$ and all $g\in X$, the mapping $T_g(Q)$ is $\mathbf{C}$-superadditive. Positive homogeneity of each $T_g(Q)$ follows from the positive homogeneity of $Q$ and $A_Q$ (see the previous point) and the fact that $\mathbf{C}(x)$ is a convex cone.

\vspace{0.2cm}

\textbf{(3)}. Fix $Q\in \mathcal{D}_{\mathrm{PH}}$. It is immediate to see that $T_g(Q)\geq Q$ (take $\lambda=0$ and $h=x$, recalling that $\textbf{C}$ is reflexive). We show that $T_g(Q)\leq H$. First notice that, if $\lambda=0$, then for all $x\in X$, by the monotonicity of $H$,
\begin{align*}
&\sup\limits_{h}\left\lbrace Q(h)+\lambda A_Q(g):h+\lambda g\preceq x,\ h\in \mathbf{C}(x) \right\rbrace=\sup\limits_{h}\left\lbrace Q(h):h\preceq x,\ h\in \mathbf{C}(x) \right\rbrace\\
&\leq \sup\limits_{h}\left\lbrace H(h):h\preceq x,\ h\in \mathbf{C}(x) \right\rbrace\leq H(x).
\end{align*}
Therefore, we can focus on all $\lambda>0$. In particular, we have that for all $x\in X$,
    \begin{align*}
        &\sup\limits_{h,\lambda}\left\lbrace Q(h)+\lambda A_Q(g):h+\lambda g\preceq x,\ h\in \mathbf{C}(x),\ \lambda> 0 \right\rbrace\\
        &=\sup\limits_{h,\lambda}\inf\limits_{y\in \mathbf{C}(x),Q(y)>-\infty}\left\lbrace Q(h)+\lambda H(g+y)-\lambda Q(y):h+\lambda g\preceq x,\ h\in \mathbf{C}(x),\ \lambda> 0 \right\rbrace\\
        &=\sup\limits_{h,\lambda}\inf\limits_{y\in \mathbf{C}(x),Q(y)>-\infty}\left\lbrace Q(h)+ H(\lambda g+\lambda y)- Q(\lambda y):h+\lambda g\preceq x,\ h\in \mathbf{C}(x),\ \lambda> 0 \right\rbrace\\
        &\leq \sup \limits_{h,\lambda}\left\lbrace H(\lambda g+h):h+\lambda g\preceq x,\ h\in \mathbf{C}(x),\ \lambda> 0  \right\rbrace\\
        &= \sup \limits_{h,\lambda}\left\lbrace H(\lambda g+\lambda h):\lambda h+\lambda g\preceq x,\ h\in \mathbf{C}(x),\ \lambda> 0  \right\rbrace\\
        &\leq H(x),
    \end{align*}
    where the last two steps follow from the fact that $\mathbf{C}(x)$ is a cone and $H$ is monotone. Thus, connecting the two observations for $\lambda=0$ and all $\lambda>0$, we have that
    \[
    T_g(Q)(x) =\sup\limits_{h,\lambda}\left\lbrace Q(h)+\lambda A_Q(g):h+\lambda g\preceq x,\ h\in \mathbf{C}(x),\ \lambda\geq  0 \right\rbrace\leq H(x),
    \]
    for all $x\in X$, and hence $Q\leq T_g(Q)\leq H$.
    
\vspace{0.3cm}

\textbf{(4)}. Let $x\in X$ and $Q\in \mathcal{D}_{\mathrm{PH}}$. Then, since $Q$ is positively homogeneous, we have
\begin{align*}
T_x(Q)(x)&=\sup\limits_{h,\lambda}\left\lbrace Q(h)+\lambda A_Q(x):h+\lambda x\preceq x,\ h\in \mathbf{C}(x),\ \lambda\geq  0 \right\rbrace\\
&\geq \sup\limits_{h,n}\left\lbrace Q(h)+ \frac{n-1}{n}A_Q(x):h\preceq \frac{1}{n}x,\ h\in \mathbf{C}(x)\right\rbrace\\
&\geq \frac{1}{m}Q\left(x\right)+\frac{m-1}{m}A_{Q}(x)\\
&\geq -\frac{1}{m}\left\lvert Q\left(x\right)\right\rvert+\frac{m-1}{m}A_{Q}(x),
\end{align*}
for all $m\in \mathbb{N}$. Thus, letting $m\to \infty$, since all Dedekind complete Riesz spaces are Archimedean\footnote{We recall that a Riesz space $(W,\leq)$ is \textit{Archimedean} if whenever $\mathbf{0}_{W}\leq nx \leq y$ for all $n=1,2,\ldots$ and some $y\geq \mathbf{0}_{W}$, we have that $x = \mathbf{0}_{W}$.} (see Lemma 8.4 in \cite{AliBorder2006}), we have $T_x(Q)(x)\geq A_{Q}(x)$ for all $x\in X$.
\end{proof}

\vspace{0.2cm}

This lemma highlights the fact that for all $g\in X$, $T_g(\cdot)$ is a selfmap, as $T_g(Q)\in \mathcal{D}_{\mathrm{PH}}$ for all $Q\in \mathcal{D}_{\mathrm{PH}}$. Now we are ready to provide the proof of Theorem \ref{prop_Pataraia}.

\vspace{0.2cm}

\begin{proof}[Proof of Theorem \ref{prop_Pataraia}]
 By Lemma \ref{toolkit_lemma}, for all $g\in X$, $T_g:\mathcal{D}_{\mathrm{PH}}\to \mathcal{D}_{\mathrm{PH}}$ is inflationary. Therefore $\left(T_g\right)_{g\in X}$ is a family of inflationary functions. By Theorem \ref{pataraiaFPT} the family $\left(T_g\right)_{g\in X}$ has a common fixed point $Q^*\in \mathcal{D}_{\mathrm{PH}}$. Since for all $g\in X$ and all $Q\in \mathcal{D}_{\mathrm{PH}}$ we have $T_g(Q)\leq H$, it follows that for all $x\in X$ and $y\in \mathbf{C}(x)$ with $Q^*(y)>-\infty$,
\begin{align*}
H(x+y)-Q^*(y)&\geq T_g(Q^*)(x+y)-Q^*(y)\\
&\geq T_g(Q^*)(x)+T_g(Q^*)(y)-T_g(Q^*)(y)=T_g(Q^*)(x).
\end{align*}

\vspace{0.2cm}

\noindent Thus, $A_{Q^*}(x)\geq T_g(Q^*)(x)=Q^*(x)$ for all $x,g\in X$. Moreover, by Lemma \ref{toolkit_lemma}-\textit{(4)}, we have $T_g(Q^*)(x)=Q^*(x)=T_x(Q^*)(x)\geq A_{Q^*}(x)$, for all $x,g\in X$. Therefore, $A_{Q^*}=Q^*$. Since $A_{Q^*}$ is $\mathbf{C}$-sublinear and $Q^*$ is $\mathbf{C}$-superlinear we have that $Q^*$ is $\mathbf{C}$-linear. To conclude, since $Q^*\in \mathcal{D}_{\mathrm{PH}}$, the claim follows. 
\end{proof}

\vspace{0.2cm}
\subsubsection{Extension and Envelope Results} By Theorem \ref{prop_Pataraia}, we derive an analogous version of the Hahn-Banach Extension Theorem (see for example Theorem 1.25 in \cite{Positive_operators}).

\vspace{0.2cm}

\begin{corollary}\label{Ext_Theorem}
Let $H:X\to \mathbb{V}\cup\left\lbrace -\infty \right\rbrace$ be $\mathbf{C}$-sublinear and monotone, and $Y\subseteq X$ be a convex cone. If $\ell:Y\to \mathbb{V}\cup\left\lbrace -\infty \right\rbrace$ is $\mathbf{C}$-linear and satisfies $\ell\leq H|_{Y}$, then there exists a $\mathbf{C}$-linear map $Q:X\to \mathbb{V}\cup\left\lbrace -\infty \right\rbrace$ such that $Q\leq H$ and $\ell\leq Q|_Y$.
\end{corollary}

\vspace{0.2cm}

\begin{proof}
Define $P:X\to \mathbb{V}\cup\left\lbrace -\infty \right\rbrace$ by
\[
P(x)=\begin{cases}
\ell(x) & x\in Y,\\
-\infty & x\notin Y,
\end{cases}
\]

\vspace{0.2cm}

\noindent for all $x\in X$. Thus $P$ is $\mathbf{C}$-superadditive. Indeed, suppose that $x\mathbf{C}y$. If $x,y\in Y$, then there is nothing to prove. If $x\notin Y$ or $y\notin Y$, then $P(x)+P(y)=-\infty\leq P(x+y)$. It is immediate to see that $P$ is positively homogeneous and $P\leq H$. Therefore, by Theorem \ref{prop_Pataraia}, there exists a $\mathbf{C}$-linear map $Q:X\to \mathbb{V}\cup\left\lbrace -\infty \right\rbrace$ such that $P\leq Q\leq H$. Therefore, $\ell=P|_{Y}\leq Q|_{Y}$.
\end{proof}

\vspace{0.3cm}

Next, as a direct application of Corollary \ref{Ext_Theorem}, we provide a general envelope representation result for $\mathbf{C}$-sublinear and monotone maps. For all $x\in X$, denote by $C_x$ the convex cone generated by $x$, that is, $C_x=\left\lbrace \lambda x:\lambda>0\right\rbrace$. Moreover, for all maps $H:X\to \mathbb{V}\cup\left\lbrace -\infty\right\rbrace$, let
\[
D(H)=\left\lbrace Q:X\to \mathbb{V}\cup\left\lbrace -\infty \right\rbrace: Q\ \textnormal{is}\ \mathbf{C}\textnormal{-linear}\ \textnormal{and\ }Q\leq H \right\rbrace.
\]

\vspace{0.3cm}

\begin{corollary}\label{envelope_Theorem}
If $H:X\to \mathbb{V}\cup\left\lbrace -\infty \right\rbrace$ is $\mathbf{C}$-sublinear and monotone, then
\[
H(x)=\sup\limits_{Q\in D(H)}Q(x),\ \hbox{for\ all}\ x\in X.
\]
\end{corollary}

\vspace{0.1cm}

\begin{proof}
Let $x\in X$ and define $\ell$ over $C_x$ as $\ell(\lambda x)=\lambda H(x)$, for all $\lambda x\in C_x$ with $\lambda>0$. Notice that $\ell$ is $\mathbf{C}$-linear and positively homogeneous. Moreover, by positive homogeneity we have that $\ell(\lambda x)=\lambda H(x)=H(\lambda x)$ for all $\lambda x\in C_x$. Thus, by Corollary \ref{Ext_Theorem}, there exists a $\mathbf{C}$-linear map $Q:X\to \mathbb{V}\cup\left\lbrace -\infty \right\rbrace$ such that $Q\leq H$ and $\ell\leq Q|_{C_x}$. Moreover,
\[
H(x)=\ell(x)\leq Q(x)\leq H(x), 
\]

\vspace{0.1cm}

\noindent and so $Q(x)=H(x)$. Therefore, for all $x\in X$, there exists a $\mathbf{C}$-linear map $Q_x:X\to \mathbb{V}\cup\left\lbrace -\infty \right\rbrace$ such that $Q_x(x)=H(x)$. This implies that
\[
H(x)=\sup\limits_{Q\in D(H)}Q(x),
\]
for all $x\in X$.
\end{proof}

\vspace{0.2cm}

\begin{remark}\label{remark_on_continuity}
Suppose that $\mathbb{V}=\mathbb{R}$ and $V$ is a topological vector space. For any map $G:V\to \left[-\infty,\infty\right)$, we define the (effective) domain of $G$ by $\textnormal{dom}G=\left\lbrace x\in V:G(x)>-\infty\right\rbrace$. We say that $G$ is proper if $\textnormal{dom}G$ is nonempty. It is important to note that Corollary \ref{envelope_Theorem} could also be proved by defining each $\ell$ on the vector space generated by $x\in V$, i.e., $\textnormal{span}\left\lbrace x\right\rbrace$, as $\ell(\alpha x)=\alpha H(x)$ for all $\alpha\in \mathbb{R}$, and then extending each $\ell$ to be equal to $-\infty$ everywhere else. In such a case, we would have that each extension $Q|_{\textnormal{dom}G}$ is continuous. Therefore, in this setting, Corollary \ref{envelope_Theorem} could be restated, adding a further restriction on $D(H)$, namely the continuity of its elements when restricted to their domains. This approach is similar to the one adopted by \cite{Roth_2000} (e.g., see Corollary 3.3).
\end{remark}

\vspace{0.1cm}
\subsubsection{Examples}\label{examples}
\begin{example}
Clearly $\mathbf{C}=X\times X$ is a convex-conic symmetric preorder. Under this convex-conic symmetric preorder $\mathbf{C}$-(super/sub)linearity corresponds to the classic (super/sub)linearity.
\end{example}

\vspace{0.2cm}

\begin{example}[{\bf{Positive homogeneity}}]
Fix $x\in X$ and define $\mathbf{D}_x=\left\lbrace (\lambda x,\beta x):\lambda,\beta>0\right\rbrace$. Then $\mathbf{D}_x$ is a convex-conic symmetric preorder. More importantly, define $\mathbf{D}=\bigcup_{x\in X}\mathbf{D}_x$. We now verify that $\mathbf{D}$ is also a convex-conic symmetric preorder. Since each $\mathbf{D}_x\subseteq \mathbf{D}$, we have that $\mathbf{D}$ is symmetric, reflexive, and satisfies property \ref{item1:poshom_ccsp}. Fix arbitrarily $x,y,z\in X$ and suppose that $x\mathbf{D}y\mathbf{D}z$. Then there exist $\alpha,\beta,\lambda,\gamma>0$ and $v,w\in X$ such that,
\[
(x,y)=(\alpha v,\beta v)\ \textnormal{and}\ (y,z)=(\lambda w,\gamma w).
\]
Therefore, $x=\alpha v=\frac{\alpha}{\beta} y=\lambda\frac{\alpha}{\beta} w$, and hence $x\mathbf{D}z$. Consequently, $\mathbf{D}$ is transitive. Moreover, 
\[
(x+y,z)=\left(\lambda\left(\frac{\alpha}{\beta}+1\right)w,\gamma w\right),
\]
and thus $\mathbf{D}$ satisfies property \ref{item2:add_ccsp}. Note that any positively homogeneous map $F:X\to \mathbb{R}$ is $\mathbf{D}$-linear.
\end{example}

\vspace{0.2cm}

\begin{example}[{\bf{Equivalent measures}}]
Consider a measurable space $(\Omega,\mathcal{F})$ and denote by $X$ the convex cone of countably additive measures $\mu:\mathcal{F}\to \left[0,\infty\right]$. For all measures $\mu\in X$ we define $N_{\mu}=\left\lbrace A\in \mathcal{F}:\mu(A)=0\right\rbrace$, i.e., the collection of $\mu$-null elements of $\mathcal{F}$. Two measures $\mu,\nu\in X$ are said to be \textit{equivalent}, denoted by $\mu\sim \nu$, if $N_{\mu}=N_{\nu}$. It can be verified that $\sim$ is a convex-conic symmetric preorder. Indeed, letting $\lambda>0$ and $\mu\in X$, we have $N_{\lambda \mu}=N_{\mu}$. If $\nu,\mu,\eta\in X$, $\nu\sim \mu$, and $\mu\sim \eta$, then $N_{\mu}=N_{\nu}=N_{\eta}$. Thus $\sim$ is transitive. Symmetry is also straightforward. To conclude, if $\nu\sim \mu\sim \eta$ for $\nu,\mu,\eta\in X$, we have that $\nu(A)+\mu(A)=0$ for all $A\in N_{\eta}$, and the converse holds as well. This yields $N_{\nu+\mu}=N_{\eta}$, and hence $\sim$ is a convex-conic symmetric preorder. 
\end{example}

\vspace{0.2cm}

\begin{example}[{\bf{Strict comonotonicity}}]\label{example_comon}
Let $(\Omega,\mathcal{F})$ be a measurable space, and denote by $L^0(\Omega,\mathcal{F})$ the space of $\mathcal{F}$-measurable real-valued functions. For all $x\in L^0(\Omega,\mathcal{F})$, denote by $C_x$ the convex cone generated by $x$. We say that $x,y\in L^0(\Omega,\mathcal{F})$ are \textit{strictly comonotonic} if either (1) $x\in C_y$; or (2) for all $\omega_1\neq \omega_2$ in $\Omega$, 
\[
\left[x(\omega_1)-x(\omega_2)\right]\left[y(\omega_1)-y(\omega_2)\right]>0.
\]

\vspace{0.2cm}

\noindent The strict comonotonicity relation, denoted as $\upharpoonleft\upharpoonright$, requires that $x \upharpoonleft\upharpoonright y$ if and only if $x,y\in L^0(\Omega,\mathcal{F})$ are strictly comonotonic. We now verify that $\upharpoonleft\upharpoonright$ is a convex-conic symmetric order. Clearly, $(\lambda x)\upharpoonleft\upharpoonright x$, for all $\lambda>0$, as $\lambda x\in C_x$, for all $x\in L^0(\Omega,\mathcal{F})$. Symmetry is immediate, since if $x\in C_y$, then $y\in C_x$. To show transitivity, fix $\omega_1\neq \omega_2$, and suppose that $x \upharpoonleft\upharpoonright y$ and $ y\upharpoonleft\upharpoonright z$. Then, we have three cases:
\vspace{0.2cm}
\begin{enumerate}
\item If $x\in C_y$ and $y\in C_z$, then $x\in C_z$. Thus, $x \upharpoonleft\upharpoonright z$.
\vspace{0.3cm}
\item If $x\in C_y$ and $y\notin C_z$, then 
\[
\left[y(\omega_1)-y(\omega_2)\right]\left[z(\omega_1)-z(\omega_2)\right]>0\ \textnormal{ and }\ x=\lambda y,
\]
for some $\lambda>0$. This implies that
\[
\left[x(\omega_1)-x(\omega_2)\right]\left[z(\omega_1)-z(\omega_2)\right]>0.
\]
Thus, $x \upharpoonleft\upharpoonright z$
\vspace{0.3cm}
\item If $x\notin C_y$ and $y\notin C_z$, then
\[
\left[x(\omega_1)-x(\omega_2)\right]\left[y(\omega_1)-y(\omega_2)\right]>0\ \textnormal{ and }\ \left[ y(\omega_1)-y(\omega_2)\right]\left[z(\omega_1)-z(\omega_2)\right]>0.
\]
\noindent Therefore, if $x(\omega_1)-x(\omega_2)>0$, then $y(\omega_1)-y(\omega_2)>0$, and hence $z(\omega_1)-z(\omega_2)>0$. The same reasoning (with inverted signs) applies to the case where $x(\omega_1)-x(\omega_2)<0$. Thus $x \upharpoonleft\upharpoonright z$.
\end{enumerate}

\vspace{0.3cm}

\noindent To conclude, fix $\omega_1\neq \omega_2$ and suppose that $x \upharpoonleft\upharpoonright y\upharpoonleft\upharpoonright z$. Then, we have three cases:
\vspace{0.2cm}
\begin{enumerate}
\item If $x\in C_y$ and $y\in C_z$, then $x+y\in C_z$. Thus, $(x+y) \upharpoonleft\upharpoonright z$.
\vspace{0.3cm}
\item If $x\in C_y$ and $y\notin C_z$, then 
\[
\left[y(\omega_1)-y(\omega_2)\right]\left[z(\omega_1)-z(\omega_2)\right]>0\ \textnormal{ and }\ x=\lambda y,
\]
for some $\lambda>0$. This implies that
\[
\left[x(\omega_1)+y(\omega_1)-x(\omega_2)-y(\omega_2)\right]\left[z(\omega_1)-z(\omega_2)\right]>0.
\]
Thus, $(x+y) \upharpoonleft\upharpoonright z$.
\vspace{0.3cm}
\item If $x\notin C_y$ and $y\notin C_z$, then 
\begin{align*}
&\left[x(\omega_1)+y(\omega_1)-x(\omega_2)-y(\omega_2)\right]\left[z(\omega_1)-z(\omega_2)\right]\\ 
&=\left[x(\omega_1)-x(\omega_2)\right]\left[z(\omega_1)-z(\omega_2)\right]+\left[y(\omega_1)-y(\omega_2)\right]\left[z(\omega_1)-z(\omega_2)\right]>0.
\end{align*}
Thus, $(x+y)\upharpoonleft\upharpoonright z$.
\end{enumerate}

\vspace{0.3cm}

\noindent Therefore, $\upharpoonleft\upharpoonright$ is a convex-conic symmetric preorder.
\end{example}

\vspace{0.2cm}

Comonotonicity is an important property in in decision theory, risk measurement, and the theory of risk sharing. A different example of a convex-conic preorder concerns vectors that are affinely related.

\vspace{0.2cm}

\begin{example}[{\bf{Affinity}}]
Let $X$ be a vector space, and fix $e\in X$. We say that $x$ and $y$ in $X$ are $e$-\textit{affinely related} if there exist $\alpha,\beta\neq 0$ such that $x=\alpha y+\beta e$. In particular, we define the binary relation $\mathbf{Aff}$ as follows
\[
x\mathbf{Aff}y \, \Longleftrightarrow \, \exists \, \alpha\neq 0, \exists \ \beta\in \mathbb{R}: x=\alpha y+\beta e.
\]

\vspace{0.1cm}

\noindent Even if straightforward, we provide the steps proving that $\mathbf{Aff}$ is a convex-conic symmetric preorder. Reflexivity is immediate, as $x=x$. Passing to symmetry, if $x\mathbf{Aff}y$, then there exist  $\alpha\neq 0$ and $\beta\in \mathbb{R}$ such that $x=\alpha y+\beta e$, and hence $y=\frac{1}{\alpha}-\beta e$, implying that $y\mathbf{Aff}x$. Suppose that $x\mathbf{Aff}y$ and $y\mathbf{Aff}z$. There exist $\alpha_1,\alpha_2\neq 0$ and $\beta_1,\beta_2\in \mathbb{R}$ such that $x=\alpha_1y+\beta_1 e$ and $y=\alpha_2z+\beta_2 e$. Then
\[
x=\alpha_1(\alpha_2z+\beta_2e)+\beta_1e=\alpha_1\alpha_2z+(\alpha_1\beta_2+\beta_1)e,
\]

\vspace{0.1cm}

\noindent proving that $x\mathbf{Aff}z$, and so $\mathbf{Aff}$ is transitive. Now fix $x\in X$, then clearly $(\lambda x)\mathbf{Aff}x$ for all $\lambda> 0$. In conclusion, if $x\mathbf{Aff}y\mathbf{Aff}z$, then there exist $\alpha_1,\alpha_2\neq 0$ and $\beta_1,\beta_2\in \mathbb{R}$ such that $x=\alpha_1y+\beta_1 e$ and $y=\alpha_2z+\beta_2 e$. Then
\begin{align*}
x+y&=\alpha_1y+\beta_1e+\alpha_2z+\beta_2e=\alpha_1\alpha_2z+\alpha_1\beta_2e+\beta_1e+\alpha_2z+\beta_2e\\    &=\left(\alpha_1\alpha_2+\alpha_2\right)z+\left(\alpha_1\beta_2+\beta_1+\beta_2\right)e,
\end{align*}

\vspace{0.1cm}

\noindent implying that $(x+y)\mathbf{Aff}z$. Thus $\mathbf{Aff}$ is a convex-conic symmetric preorder.
\end{example}

\vspace{0.2cm}

 To conclude this section we provide some examples of binary relations that are \textit{not} convex-conic symmetric preorders. Clearly any irreflexive, asymmetric, or nontransitive binary relation would work. For instance, probabilistic independence and orthogonality in the context of inner product spaces are relevant examples.

\vspace{0.2cm}
 
\begin{example}[{\bf{Inner product spaces}}]
Consider the Hilbert space $L^{2}\left(\left[0,1\right],\mathcal{B}\left[0,1\right],\mathrm{Leb}\right)$, where $\mathcal{B}\left[0,1\right]$ denotes the Borel sigma algebra and $\mathrm{Leb}$ the Lebesgue measure. We define the relation $\mathbf{Corr}$ by
\[
f\textnormal{\textbf{Corr}}g \Longleftrightarrow \int_{\left[0,1\right]}fg \, \textnormal{d}\mathrm{Leb}\geq 0.
\]

\vspace{0.1cm}

\noindent It is easy to see that $\mathbf{Corr}$ is reflexive and symmetric, and it satisfies conditions \ref{item1:poshom_ccsp} and \ref{item2:add_ccsp} of Definition \ref{def:ccsp}, but it fails transitivity. Indeed, consider the functions
\[
f=\mathbf{1}_{\left[0,1\right]},\ g=\mathbf{1}_{\left[0,\frac{1}{2}\right]},\ h=-\mathbf{1}_{\left[\frac{1}{2},1\right]}.
\]
Then,
\[
f\mathbf{Corr}g\ \textnormal{and}\ g \mathbf{Corr}h,\ \textnormal{but}\ \int_{\left[0,1\right]}fh\,\textnormal{d}\mathrm{Leb}=-\frac{1}{2}<0.
\]
\end{example}

\vspace{0.2cm}

The next example provides a symmetric preorder that fails property \ref{item2:add_ccsp} of Definition \ref{def:ccsp}.

\vspace{0.2cm}

\begin{example}
Define the function $\varphi:\mathbb{R}^2\to \mathbb{R}$ by
\[
\varphi(x,y)=\begin{cases}
-1 & (x,y)\in \left(\left(\mathbb{R}\setminus \left\lbrace0 \right\rbrace\right) \times \left\lbrace 0 \right\rbrace\right)\cup \left(\left\lbrace 0\right\rbrace\times \left(\mathbb{R}\setminus \left\lbrace 0\right\rbrace\right)\right)\\
0& \textnormal{otherwise}.
\end{cases}
\]

\vspace{0.2cm}

\noindent Consider the binary relation $x\mathlarger{\mathlarger{\boldsymbol{\varphi}}}y$ if and only if $\varphi(x,y)\geq 0$. By definition, $\varphi(x,x)=0$ for all $x\in \mathbb{R}$, and so $\mathlarger{\mathlarger{\boldsymbol{\varphi}}}$ is reflexive. Moreover, $\varphi(x,y)=\varphi(y,x)$ for all $x,y\in \mathbb{R}$, proving that $\mathlarger{\mathlarger{\boldsymbol{\varphi}}}$ is symmetric. Additionally, if $x,y,z\in \mathbb{R}$ with $\varphi(x,y)\geq 0$ and $\varphi(y,z)\geq 0$, then it must be the case that either $x,y,z=0$ or $x,y,z\neq 0$, and hence, $\varphi(x,z)=0$. Therefore $\mathlarger{\mathlarger{\boldsymbol{\varphi}}}$ is transitive. However, letting $x,z=1,y=-1$, we obtain
\[
\varphi(x,y)\geq 0,\ \varphi(y,z)\geq 0,\ \textnormal{and}\ \varphi(x+y,z)=\varphi(0,z)<0.
\]

\vspace{0.2cm}

\noindent Thus $\mathlarger{\mathlarger{\boldsymbol{\varphi}}}$ is a symmetric preorder that is not a convex-conic symmetric preorder.
\end{example}

\vspace{0.4cm}
\section{Relaxing positive homogeneity}
In this section, we weaken the definition of a convex-conic symmetric preorder by removing property \ref{item1:poshom_ccsp} (i.e., positive homogeneity). In particular, we define a new preorder called \say{summand symmetric preorder}. Fix a vector space $V$ and a subset $X\subseteq V$ such that $X+X\subseteq X$. 

\vspace{0.1cm}

\begin{defi}\label{def:summandp}
A binary relation $\mathbf{S}\subseteq X\times X$ is a \textit{summand symmetric preorder} if it is reflexive, transitive, and symmetric, and it satisfies the following property:
\vspace{0.1cm}
\begin{equation}
 \label{item2:add_summand preorder} x\mathbf{S}y\mathbf{S}z\ \textnormal{implies}\ (x+y)\mathbf{S}z,\ \textnormal{for\ all}\ x,y,z\in X.
\end{equation}
\end{defi}

\vspace{0.2cm}

Note that in this case $\mathbf{S}(x)=\left\lbrace y\in X:y\mathbf{S}x \right\rbrace$ is closed with respect to the summation of its elements. Clearly, if $X$ is a convex cone and $\mathbf{C}$ is a convex-conic symmetric preorder on $X$, then $\mathbf{C}$ is also a summand symmetric preorder. 

\vspace{0.2cm}

Let $\left(\mathbb{V},\leq \right)$ be a Dedekind-complete Riesz space. We say that a map $F:X\to\mathbb{V}\cup\left\lbrace -\infty \right\rbrace$ is $\mathbf{S}$-\textit{sublinear} if
\vspace{0.2cm}
\begin{enumerate}[label=(\roman*)]
    \item  $F$ is \textit{positively integer homogeneous}, i.e., $F(n x)=n F(x)$, for all $n\in \mathbb{N}$ and all $x\in X$.
   \vspace{-0.2cm}
    \item $F$ is $\mathbf{S}$-\textit{subadditive}, i.e., $F(x+y)\leq F(x)+F(y)$, for all $x,y\in X$ with $x\mathbf{S}y$.
\end{enumerate}

\vspace{0.3cm}

The definitions of $\mathbf{S}$-superlinearity and $\mathbf{S}$-linearity are analogous. We endow $V$ with a partial order $\preceq$ such that $(V,\preceq)$ is a partially ordered vector space. We say that a map $F:X\to \mathbb{V}\cup\left\lbrace -\infty \right\rbrace$ is \textit{monotone} if $x\preceq y$ implies $F(x)\leq F(y)$, for all $x,y\in X$. We adopt the convention that $0\cdot (-\infty)=\mathbf{0}_{\mathbb{V}}$. Using the exact same steps as in the proof of Theorem \ref{prop_Pataraia}, slightly changing the auxiliary maps, we can retrieve the following version of Theorem \ref{prop_Pataraia} for $\mathbf{S}$-(super/sub)linear functionals.

\vspace{0.2cm}

\begin{theorem}\label{prop_Pataraia_summand}
Suppose that $X$ is a convex cone, $P:X\to \mathbb{V}\cup\left\lbrace -\infty \right\rbrace$ is $\mathbf{S}$-superlinear, and $H:X\to \mathbb{V}\cup\left\lbrace -\infty \right\rbrace$ is $\mathbf{S}$-sublinear and monotone. If $P\leq H$, then there exists an $\mathbf{S}$-linear map $Q:X\to \mathbb{V}\cup\left\lbrace -\infty \right\rbrace$ such that $P\leq Q\leq H$. 
\end{theorem}

\vspace{0.2cm}

We omit the proof of this result as it is almost identical to that of Theorem \ref{prop_Pataraia}. We simply provide the form of one of the auxiliary maps we used in the previous proofs. In particular, let 
\[\mathcal{D}_{\mathrm{PH}}=\left\lbrace Q:X\to \mathbb{V}\cup\left\lbrace -\infty \right\rbrace:Q\ \textnormal{is}\ \mathbf{S}\textnormal{-superlinear} \ \hbox{and} \  P\leq Q\leq H \right\rbrace.\]

\vspace{0.2cm}

\noindent Now, fix  $g,x\in X$, and $Q\in \mathcal{D}_{\mathrm{PH}}$. We also define,
\[
T_g(Q)(x)=\sup\limits_{h,n}\left\lbrace Q(h)+n A_Q(g):h+n g\preceq x,\ h\in \mathbf{S}(x),\ n\in \mathbb{N}\cup \left\lbrace 0\right\rbrace \right\rbrace.
\]

\vspace{0.2cm}

\noindent The auxiliary function $A_Q$ is defined as in Section \ref{sandwich_ccsp}, using $\mathbf{S}$ instead of $\mathbf{C}$. 

\vspace{0.3cm}

The analogous of Lemma \ref{toolkit_lemma} also holds in this setting. Before providing its proof, we require the following simple result.

\vspace{0.2cm}

\begin{lemma}\label{mini_lemma}
Suppose that $X$ is a convex cone and $\mathbf{S}$ a summand symmetric preorder on $X$. Then for all $x,y\in X$, 
$$x\mathbf{S}y   \implies  \frac{x}{n}\mathbf{S}y\ \hbox{for\ all}\ n\in \mathbb{N}.$$ 
\end{lemma}

\vspace{0.2cm}

\begin{proof}
Fix $x,y\in X$ with $x\mathbf{S}y$, and choose $n\in \mathbb{N}$ arbitrarily. Since $\mathbf{S}$ is reflexive, it follows that $x/n\mathbf{S}x/n$. Hence, by property \eqref{item2:add_summand preorder}, we obtain
\[
x=\left(\sum_{k=1}^n\frac{x}{n}\right)\mathbf{S}\frac{x}{n}.
\]

\vspace{0.1cm}

\noindent Thus, by transitivity and symmetry of $\mathbf{S}$, we have $x/n\mathbf{S}y$. 
\end{proof}

\vspace{0.1cm}

\begin{lemma}\label{toolkit_lemma_summand}
The following claims hold:
\vspace{0.2cm}
\begin{enumerate}
    \item For all $Q\in \mathcal{D}_{\mathrm{PH}}$, we have $A_Q\geq Q$ and $A_Q$ is $\mathbf{S}$-sublinear.
    \vspace{0.3cm}
    \item For all $Q\in \mathcal{D}_{\mathrm{PH}}$ and $g\in X$, we have $T_g(Q)$ is $\mathbf{S}$-superlinear.
    \vspace{0.3cm}
    \item\label{item:3} For all $Q\in \mathcal{D}_{\mathrm{PH}}$, we have $Q\leq T_g(Q)\leq H$. 
          \vspace{0.3cm}
    \item\label{item:4} For all $Q\in \mathcal{D}_{\mathrm{PH}}$ and all $x\in X$, we have $T_x(Q)(x)\geq A_Q(x)$.
\end{enumerate}
\end{lemma}

\vspace{0.2cm}

\begin{proof}
We provide a proof only for points (\ref{item:3}) and (\ref{item:4}). For the remaining points, the proofs are identical to those of Lemma \ref{toolkit_lemma}. 
\vspace{0.3cm}

\textbf{(3)}.
It is immediate to see that $T_g(Q)\geq Q$, for all $Q\in \mathcal{D}_{\mathrm{PH}}$ (take $n=0$ and $h=x$). We show that $T_g(Q)\leq H$. In particular, we have that for all $x\in X$,
\vspace{0.2cm}
    \begin{align*}
        T_g(Q)(x)&=\sup\limits_{h,n}\left\lbrace Q(h)+n A_Q(g):h+n g\preceq x,\ h\in \mathbf{S}(x),\ n\in \mathbb{N}\cup \left\lbrace 0\right\rbrace \right\rbrace\\
        &=\sup\limits_{h,n}\inf\limits_{y\in \mathbf{S}(x),Q(y)>-\infty}\left\lbrace Q(h)+n H(g+y)-n Q(y):h+n g\preceq x,\ h\in \mathbf{S}(x),\ n\in \mathbb{N}\cup \left\lbrace 0\right\rbrace \right\rbrace\\
        &=\sup\limits_{h,n}\inf\limits_{y\in \mathbf{S}(x),Q(y)>-\infty}\left\lbrace Q(h)+ H(n g+n y)- Q(n y):h+n g\preceq x,\ h\in \mathbf{S}(x),\ n\in \mathbb{N}\cup \left\lbrace 0\right\rbrace \right\rbrace\\
        &\leq \sup \limits_{h,n}\left\lbrace H(n g+h):h+n g\preceq x,\ h\in \mathbf{S}(x),\ n\in \mathbb{N}\cup \left\lbrace 0\right\rbrace  \right\rbrace\\
        &= \sup \limits_{h,n}\left\lbrace H(n g+n h):n h+n g\preceq x,\ h\in \mathbf{S}(x),\ n\in \mathbb{N}\cup \left\lbrace 0\right\rbrace  \right\rbrace\\
        &\leq H(x),
    \end{align*}

\vspace{0.1cm}

\noindent where the fourth inequality follows from setting $y=\frac{h}{n}$ (which belongs to $\mathbf{S}(x)$ by Lemma \ref{mini_lemma}), while the last two steps follow from the fact that $\mathbf{S}(x)$ is closed with respect to addition and $H$ is monotone. Thus, $Q\leq T_g(Q)\leq H$.

\vspace{0.3cm}

\textbf{(4)}. Let $x\in X$, $Q\in \mathcal{D}_{\mathrm{PH}}$, and $m\in \mathbb{N}$. Then, since $Q$ and $T_x(Q)$ are integer positively homogeneous, we have
\begin{align*}
T_x(Q)(x)&=\frac{T_x(Q)((m+1)x)}{m+1}\\
&=\frac{1}{m+1}\sup\limits_{h,n}\left\lbrace Q(h)+n A_Q(x):h+n x\preceq (m+1)x,\ h\in \mathbf{S}(x),\ n\in \mathbb{N}\cup \left\lbrace 0\right\rbrace \right\rbrace\\
&\geq \frac{1}{m+1}\sup\limits_{h}\left\lbrace Q(h)+m A_Q(x):h\preceq x,\ h\in \mathbf{S}(x) \right\rbrace\\
&\geq \frac{1}{m+1}Q\left(x\right)+\frac{m}{m+1}A_{Q}(x)\\
&\geq -\frac{1}{m+1}\left\lvert Q\left(x\right)\right\rvert+\frac{m}{m+1}A_{Q}(x).
\end{align*}

\vspace{0.1cm}

\noindent Thus, letting $m\to \infty$, since all Dedekind complete Riesz spaces are Archimedean (see Lemma 8.4 in \cite{AliBorder2006}), we have $T_x(Q)(x)\geq A_{Q}(x)$, for all $x\in X$.
\end{proof}

\vspace{0.2cm}

The proof of Theorem \ref{prop_Pataraia_summand} is now totally analogous to that of Theorem \ref{prop_Pataraia}, and it is omitted.



\vspace{0.4cm}
\section{Applications: Comonotonic Subadditivity}\label{sect:applications}

In this section, we apply the previous results to the case of comonotonicity on a specific measurable space. Suppose that  $\Omega=\left[0,1\right]$ and $\mathcal{F}=\mathcal{B}\left[0,1\right]$, the Borel sigma algebra on the unit interval. Using our results and observations in Example \ref{example_comon}, we obtain the following.

\vspace{0.2cm}

\begin{proposition}\label{com_repre}
Suppose that $H:B(\Omega,\mathcal{F})\to \left[-\infty,\infty\right)$ is $\lVert\cdot\rVert_{\infty}$-continuous when restricted to its domain, monotone, positively homogeneous, and strictly comonotonic subadditive. Then 
\begin{equation}\label{enve_com}
H(x)=\sup\limits_{Q\in D(H)}Q(x),\ \hbox{for\ all} \ x\in B(\Omega,\mathcal{F}),
\end{equation}
where $D(H)$ is a set of maps from $B(\Omega,\mathcal{F})$ to $\left[-\infty,\infty\right)$ that are comonotonic additive and $\lVert\cdot\rVert_{\infty}$-continuous when restricted to their domains. Moreover, $H$ is comonotonic subadditive on its domain.
\end{proposition}

\vspace{0.1cm}

\begin{proof}
By Corollary \ref{envelope_Theorem}, Remark \ref{remark_on_continuity}, and Example \ref{example_comon}, it follows that \eqref{enve_com} holds, where $D(H)$ is a convex set of functionals from $B(\Omega,\mathcal{F})$ to $\left[-\infty,\infty\right)$ that are strictly comonotonic additive and $\lVert\cdot\rVert_{\infty}$-continuous when restricted to their domains. By Lemma \ref{lemma_com_reals} in the \hyperlink{LinkToAppendix}{Appendix}, all elements of $D(H)$ are comonotonic additive on their domains. Note that Lemma \ref{lemma_com_reals} also implies that $H$ is comonotonic subadditive on its domain.
\end{proof}

\vspace{0.2cm}

Comonotonic additive functionals play a central role in the theory of decision-making under ambiguity. Their study was pioneered by \cite{Schmeidler_89}, who also provided a representation of such functionals in terms of Choquet integrals (\cite{Schmeidler86}). Such functionals are also relevant for their connection to the theory of risk measurement (e.g.,  see \cite{Denuitetal2005} or \cite{FollmerSchied2016}). As for comonotonic subadditivity, less attention has been devoted to this property. \cite{Song2006} provided, along with further properties, a full characterization of comonotonic subadditive functionals as envelopes of Choquet integrals, and \cite{Song2009} provided some applications thereof.

\vspace{0.3cm}

\begin{remark}
Proposition \ref{com_repre} hints towards the possibility of representing continuous, monotone, positively homogeneous, and strictly comonotonic subadditive functionals as suprema of signed Choquet integrals. Indeed, as shown by \cite{Wang_com}, comonotonic additive and continuous functionals can be represented by Choquet integrals with respect to signed capacities. 
\end{remark}

\vspace{0.4cm}
\section{Concluding Remarks and An Open Question}
Building upon the work of \cite{Amarante_positivity} and \cite{Isotone_Wright}, we provided a nonlinear version of the classical Sandwich Theorem. We used this result to retrieve an extension result and an envelope representation result. Our examples show that the type of nonlinearities that we introduce include some important cases that have been devoted considerable attention in decision theory and mathematical finance. Our approach highlights the possibility of retrieving Hahn-Banach-type extension results that are widely applied in functional analysis, theoretical economics, and mathematical finance. We conclude with an (informal) question and two conjectures:

\vspace{0.2cm}

\begin{enumerate}[start=1,label={(\bfseries O\arabic*)}]
\item To what extent do our results depend on the fact that our maps can take $-\infty$ as a value? More formally, is it possible to prove the following reformulations of Corollaries \ref{Ext_Theorem} and \ref{envelope_Theorem}?

\vspace{0.2cm}

\begin{conj}\label{conj_1}
Let $H:X\to \mathbb{V}$ be $\mathbf{C}$-sublinear and monotone, and let $Y\subseteq X$ be a convex cone. If $\ell:Y\to \mathbb{V}$ is $\mathbf{C}$-linear and satisfies $\ell\leq H|_{Y}$, then there exists a $\mathbf{C}$-linear $Q:X\to \mathbb{V}$ such that $Q\leq H$ and $\ell\leq Q|_Y$.
\end{conj}

\vspace{0.2cm}

\begin{conj}\label{conj_2}
If $H:X\to \mathbb{V}$ is $\mathbf{C}$-sublinear and monotone, then
\[
H(x)=\sup\limits_{Q\in \tilde{D}(H)}Q(x), \ \hbox{for\ all} \ x\in X,
\]
where
\[
\tilde{D}(H)=\Big\{Q:X\to \mathbb{V} \ \big\vert \ Q\ \textnormal{is}\ \mathbf{C}\textnormal{-linear}\ \textnormal{and\ }Q\leq H \Big\}.
\]
\end{conj}
\end{enumerate}

\newpage

\setlength{\parskip}{0.5ex}
\hypertarget{LinkToAppendix}{\ }
\appendix

\vspace{-1cm}

\section{Additional Results for Section \ref{sect:applications}}
In this section we report some auxiliary results that we applied in Section \ref{sect:applications}. In particular, we provide a partial answer to the following 

\vspace{0.2cm}

\begin{question}\label{question_comono}
If $x,y\in B(\Omega,\mathcal{F})$ are comonotonic, then there exist two sequences $(x_n)_{n\in \mathbb{N}},(y_n)_{n\in \mathbb{N}}$ in $B(\Omega,\mathcal{F})$ such that:
\begin{enumerate}
\item $x_n\to x$ and $y_n\to y$;
\vspace{0.2cm}
\item for all $n\in \mathbb{N}$, $x_n$ and $y_n$ are strictly comonotonic.
\end{enumerate}
\end{question}

\vspace{0.2cm}

It is immediate to see that this does not hold in general measurable spaces. Indeed, take any $\Omega\neq\emptyset$ and let $\mathcal{F}=\left\lbrace \emptyset,\Omega \right\rbrace$. In this measurable space $(\Omega,\mathcal{F})$, a function is $\mathcal{F}$-measurable if and only if it is constant, and therefore there is no injective measurable function. This observation highlights the fact that Question \ref{question_comono} may admit a positive answer only if we focus on measurable spaces with a sufficiently sparse sigma-algebra, where this sparsity depends also on the cardinality of $\Omega$. Providing a full answer to this question is out of the scope of this paper. Hence, we focus on a special case. However, we first 
need some auxiliary lemmas.

\vspace{0.2cm}

\begin{lemma}\label{charact_strict_comonot}
Let $x,y\in B(\Omega,\mathcal{F})$. The following are equivalent
\begin{enumerate}[label=(\roman*)]
\item $x$ and $y$ are strictly comonotonic with $x\notin C_y$.
\vspace{0.2cm}
\item there exist two increasing functions $h,g:\mathbb{R}\to \mathbb{R}$ and an injective $z\in B(\Omega,\mathcal{F})$ such that $x=h(z)$, $y=g(z)$, and $h,g$ are injective over $z(\Omega)$.
\end{enumerate}
\end{lemma}

\vspace{0.2cm}

\begin{proof}
$(i)\Rightarrow (ii)$. Since $x,y$ are comonotonic, there exist increasing functions $h,g:\mathbb{R}\to \mathbb{R}$ and $z\in B(\Omega,\mathcal{F})$ such that $x=h(z)$ and $y=g(z)$ (see e.g., \cite{Comonot_DDGR}, Theorem 2.7). Suppose that $z$ is not injective. Then, there exist $\omega_1,\omega_2\in \Omega$ such that $\omega_1\neq \omega_2$ and $z(\omega_1)=z(\omega_2)$. Thus, 
\[
\left[x(\omega_1)-x(\omega_2)\right]\left[y(\omega_1)-y(\omega_2)\right]=\left[h(z(\omega_1))-h(z(\omega_2))\right]\left[y(\omega_1)-y(\omega_2)\right]=0,
\]
contradicting the strict comonotonicity of $x,y$. Therefore, $z$ must be injective. Now suppose that $h$ is not injective over $z(\Omega)$. Then there exist  $\omega_1,\omega_2\in \Omega$ such that $\omega_1\neq \omega_2$, $z(\omega_1)\neq z(\omega_2)$, and $h(z(\omega_1))=h(z(\omega_2))$. Since $z$ is injective, we have that $\omega_1\neq \omega_2$ and
\[
\left[x(\omega_1)-x(\omega_2)\right]\left[y(\omega_1)-y(\omega_2)\right]=\left[h(z(\omega_1))-h(z(\omega_2))\right]\left[y(\omega_1)-y(\omega_2)\right]=0,
\]
contradicting the strict comonotonicity of $x,y$. Interchanging $h$ with $g$ and $x$ with $y$, the same conclusion holds for $g$ as well. 

\vspace{0.2cm}

$(ii)\Rightarrow (i)$. If $\omega_1\neq \omega_2$ and $h(z(\omega_1))>h(z(\omega_2))$, then we must have that $z(\omega_1)>z(\omega_2)$ and hence $g(z(\omega_1))>g(z(\omega_2))$. This proves the claim.
\end{proof}

\vspace{0.2cm}

\begin{lemma}\label{piecewise_approx}
Suppose that $f:\left[a,b\right]\to \mathbb{R}$ is $1$-Lipschitz and increasing. Then, there exists a sequence of $1$-Lipschitz and strictly increasing functions $(f_n)_{n\in \mathbb{N}}$ from $\left[a,b\right]$ to $\mathbb{R}$ that converges uniformly to $f$.
\end{lemma}

\vspace{0.2cm}

\begin{proof}
Since $f$ is continuous and increasing, there exist at most countably many disjoint nondegenerate intervals $(I_n)_{n\in \mathbb{N}}$ over which $f$ is constant, i.e., $f(I_n)=\left\lbrace k_n\right\rbrace$ for all $n\in \mathbb{N}$ and some $k_n\in \mathbb{R}$. Fix $\varepsilon>0$ and take finitely many points, $a=a_1<\ldots<a_k=b$ such that $\left\lvert a_i-a_{i+1}\right\rvert<\varepsilon/2$ for all $i=1,\ldots,k-1$. Suppose that $f(a_i)=f(a_j)$ for some $i<j$. Since $f$ is increasing, we have that
\[
f(a_i)=f(a_{i+1})=\ldots=f(a_j).
\]

\vspace{0.2cm}

\noindent This implies that $a_i,\ldots,a_j\in I_n$ for some $n\in \mathbb{N}$. Thus, to find a strictly increasing approximation we need to modify our vector of points. In particular, we remove all points $a_i,\ldots,a_{j-1}$, and we add one point $\tilde{a}_i$ picked from the set \[\left\lbrace x\in \left[a_{i-1},\inf I_n\right):f(x)<f(a_i)\ \textnormal{and}\ \left\lvert f(x)-f(a_i)\right\rvert<\frac{\varepsilon}{2}\right\rbrace.\]
Repeating this operation for all points where the function is constant, within the set $\left\lbrace a_1,\ldots,a_k\right\rbrace$, we retrieve (in a finite amount of operations) a set $\left\lbrace \tilde{a}_1,\ldots,\tilde{a}_m \right\rbrace$ with $\tilde{a}_1=a$ and $\tilde{a}_m=b$ such that
\[
f\left(\tilde{a}_1\right)<\left(\tilde{a}_2\right)<\ldots<f(\tilde{a}_{m-1})<f(\tilde{a}_{m}).
\]

\vspace{0.2cm}

\noindent Define the function $f^m:\left\lbrace \tilde{a}_1,\ldots,\tilde{a}_m \right\rbrace\to \mathbb{R}$ by $f^m(\tilde{a}_i)=f(\tilde{a}_i)$, for all $i=1,\ldots,m$. By linear interpolation, we can extend $f^m$ to the whole interval $\left[a,b\right]$, and we denote such an extension again by $f^m$. Clearly, $f^m$ is strictly increasing, and we now show that it must also be 1-Lipschitz. To this end, notice that the slope of $f^m|_{\left[\tilde{a}_i,\tilde{a}_{i+1}\right]}$ satisfies
\[
\frac{f(\tilde{a}_{i+1})-f(\tilde{a}_i)}{\tilde{a}_{i+1}-\tilde{a}_i}\leq 1,
\]

\vspace{0.2cm}

\noindent for all $i=1,\ldots,k-1$, where the inequality follows from the fact that $f$ is $1$-Lipschitz, $\tilde{a}_{i+1}>\tilde{a}_i$, and $f$ is increasing. This implies that $f^m|_{\left[\tilde{a}_i,\tilde{a}_{i+1}\right]}$ is $1$-Lipschitz for all $i=1,\ldots,k-1$. If $x>y$, then $x\in \left(\tilde{a}_i,\tilde{a}_{i+1}\right]$ and $y\in \left[\tilde{a}_j,\tilde{a}_{j+1}\right]$ for some $i>j$. This implies that
\vspace{0.1cm}
{\small\begin{align*}
\left\lvert f^{m}(x)-f^m(y)\right\rvert &= f^{m}(x)-f^m(y)\\
&= f^m(x)-f^{m}(\tilde{a}_i)+f^m(\tilde{a}_i)-f^m(\tilde{a}_{i-1})+\ldots+f^m(\tilde{a}_{j+1})-f^m(\tilde{a}_j)+f^m(\tilde{a}_j)-f^m(y)\\
&\leq x-\tilde{a}_i+\tilde{a}_i+\ldots+\tilde{a}_{j+1}-\tilde{a}_j+\tilde{a}_j-y=\left\lvert x-y\right\rvert.
\end{align*}} 

\vspace{0.2cm}

\noindent Thus $f^m$ is $1$-Lipschitz. Now we prove that $f^m$ is $\varepsilon$-close to $f$. For all $x\in \left[\tilde{a}_{i-1},\tilde{a}_i\right]$ and $i=2,\ldots,m$ we have
\begin{align*}
\left\lvert f(x)-f^m(x)\right\rvert & \leq \left\lvert f(x)-f(\tilde{a}_{i-1})\right\rvert+\left\lvert f(\tilde{a}_{i-1})-f^m(x)\right\rvert\\
&\leq \frac{\varepsilon}{2}+\left\lvert f^m(\tilde{a}_{i-1})-f^m(\tilde{a}_{i})\right\rvert\\
&= \frac{\varepsilon}{2}+\left\lvert f(\tilde{a}_{i-1})-f(\tilde{a}_{i})\right\rvert\\
&<\varepsilon.
\end{align*}
Since $x$ was chosen arbitrarily, it follows that $\lVert f-f^m\rVert_{\infty} \to 0$.
Therefore, there exists a sequence of $1$-Lipschitz and strictly increasing functions $(f_n)_{n\in \mathbb{N}}$ converging uniformly to $f$.
\end{proof}

\vspace{0.2cm}

Now suppose that $\Omega=\left[0,1\right]$ and $\mathcal{F}=\mathcal{B}\left[0,1\right]$ is the Borel sigma-algebra.\footnote{All the results provided in this section would hold for any closed interval $I\subseteq \mathbb{R}$.}

\vspace{0.2cm}

\begin{lemma}\label{lemma_com_reals}
If $x,y\in B(\Omega,\mathcal{F})$ are comonotonic, then there exist two sequences $(x_n)_{n\in \mathbb{N}},(y_n)_{n\in \mathbb{N}}$ in $B(\Omega,\mathcal{F})$ such that:
\begin{enumerate}
\item $x_n\xrightarrow{\lVert\cdot \rVert_{\infty}} x$ and $y_n\xrightarrow{\lVert \cdot\rVert_{\infty}} y$;
\vspace{0.2cm}
\item For all $n\in \mathbb{N}$, $x_n$ and $y_n$ are strictly comonotonic with $x_n\notin C_{y_n}$.
\end{enumerate}
\end{lemma}

\vspace{0.2cm}

\begin{proof}
Since $x,y\in B(\Omega,\mathcal{F})$ are comonotonic, there exist two increasing $1$-Lipschitz functions $h,g:\mathbb{R}\to \mathbb{R}$ and an $z\in B(\Omega,\mathcal{F})$ such that $x=h(z)$, $y=g(z)$ (\cite{Comonot_DDGR}, Theorem 2.7). We first prove a claim that will yield the result.

\vspace{0.2cm}

\begin{claim}\label{inject_measur_approx}
There exists a sequence of injective and measurable functions converging to $z$.
\end{claim}
\vspace{0.1cm}
\begin{proof}[Proof of the claim.]
Let $\left(s_n\right)_{n\in \mathbb{N}}\in  B(\Omega,\mathcal{F})^{\mathbb{N}}$ be a sequence of step functions converging uniformly to $z$. Each $s_n$ can be uniquely identified with a partition $\left(I^n_i\right)_{i=1}^{k_n}$ of nondegenerate subintervals of $\Omega$, and a vector of values $\left(a^n_{1},\ldots,a^n_{k_n}\right)$, for some $k_n\in \mathbb{N}$. For all $n\in \mathbb{N}$ and $\varepsilon>0$, we can define the following,
\[
s^\varepsilon_{n}(\omega)=2\varepsilon \left(\frac{\omega-\inf I^n_i}{\sup I^n_i-\inf I^n_i}\right)+a^n_i-\varepsilon,
\]

\vspace{0.2cm}

\noindent for all $\omega \in I^n_i$ and all $i=1,\ldots,k_n$. Clearly $s^{\varepsilon}_n$ is an injective Borel measurable function, for all $\varepsilon>0$ and all $n\in \mathbb{N}$. Intuitively, we are simply rotating slightly the constant \say{lines} of each $s_n$ over all their partitions. Moreover, note that
\[
\left\lVert s^\varepsilon_{n}-s_n \right\rVert_{\infty}\leq \varepsilon,
\]

\vspace{0.2cm}

\noindent for all $n\in \mathbb{N}$ and $\varepsilon>0$. This implies that $\left(s^{1/n}_n\right)_{n\in \mathbb{N}}$ converges uniformly to $z$. Indeed,
\[
\left\lVert z-s^{1/n}_n \right\rVert_{\infty}\leq \left\lVert z-s_n \right\rVert_{\infty}+\left\lVert s_n-s^{1/n}_{n} \right\rVert_{\infty}\leq \left\lVert z-s_n \right\rVert_{\infty}+\frac{1}{n}\to 0.
\]

\vspace{0.2cm}

\noindent Thus, we found a sequence of injective and measurable functions converging uniformly to $z$. 
\end{proof}

\vspace{0.2cm}
Given that $z$ is bounded, there exist $m,M\in \mathbb{R}$ such that $z(\Omega)\subseteq \left(m,M\right)$. Since $\left(s^{1/n}_n\right)_{n\in \mathbb{N}}$ converges uniformly to $z$, there exists some $N\in \mathbb{N}$ sufficiently large so that 
\[
s_{n}^{1/n}(\Omega)\subseteq \left[m,M\right],
\]
for all $n\geq N$. Using a slight abuse of notation, we will now identify by $\left(s^{1/n}_n\right)_{n\in \mathbb{N}}$ its subsequence $\left(s^{1/n_k}_{n_k}\right)_{k\in \mathbb{N}}$ with $n_1=N$, $n_2=N+1$ and so on. By Lemma \ref{piecewise_approx}, there exist two sequences of $1$-Lipschitz and strictly increasing functions $\left(h_n\right)_{n\in \mathbb{N}}$ and $\left(g_n\right)_{n\in \mathbb{N}}$ from $\left[m,M\right]$ to $\mathbb{R}$ converging uniformly to $h$ and $g$. For all $n\in \mathbb{N}$, let $x_n=h_n\left(s_{n}^{1/n}\right)$ and $y_n=g_n\left(s_{n}^{1/n}\right)$. Fix $n\in \mathbb{N}$ arbitrarily. If $\omega_1\neq \omega_2$, then $s_n^{1/n}(\omega_1)\neq s_n^{1/n}(\omega_2)$ since $s_n^{1/n}$ is injective, say without loss of generality that $s_n^{1/n}(\omega_1)> s_n^{1/n}(\omega_2)$. Since $h_n$ and $g_n$ are both strictly increasing we have that
\vspace{0.1cm}
\begin{align*}
&\left[x_n(\omega_1)-x_n(\omega_2)\right]\left[y_n(\omega_1)-y_n(\omega_2)\right]\\&=\left[h_n\left(s_{n}^{1/n}\right)(\omega_1)-h_n\left(s_{n}^{1/n}\right)(\omega_2)\right]\left[g_n\left(s_{n}^{1/n}\right)(\omega_1)-g_n\left(s_{n}^{1/n}\right)(\omega_2)\right]>0.
\end{align*}

\vspace{0.2cm}

\noindent Thus, $x_n,y_n$ are strictly comonotonic. Fix $n\in \mathbb{N}$ arbitrarily. Since $h_n$ is $1$-Lipschitz, we have
\vspace{0.1cm}
\begin{align*}
\left\lVert h(z)-h_n\left(s_{n}^{1/n}\right)\right\rVert_{\infty}&\leq \left\lVert h(z)-h_n\left(z\right)\right\rVert_{\infty}+ \left\lVert h_n(z)-h_n\left(s_{n}^{1/n}\right) \right\rVert_{\infty} \\
&=\left\lVert h(z)-h_n\left(z\right)\right\rVert_{\infty}+\sup\limits_{\omega\in \Omega}\left\lvert h_n(z(\omega))-h_n\left(s_{n}^{1/n}(\omega)\right) \right\rvert\\
&\leq \left\lVert h(z)-h_n\left(z\right)\right\rVert_{\infty}+\sup\limits_{\omega\in \Omega}\left\lvert z(\omega)-s_{n}^{1/n}(\omega)\right\rvert\to 0.
\end{align*}

\vspace{0.2cm}

\noindent Thus $\left(x_n\right)_{n\in \mathbb{N}}$ converges uniformly to $x$. The same holds for $(y_n)_{n\in \mathbb{N}}$, and the proof is totally analogous.
\end{proof}

\vspace{0.8cm}

\bibliographystyle{apalike}
\bibliography{bibliography}

\vspace{0.7cm}

\end{document}